\documentclass[12pt,reqno]{article}

\usepackage[usenames]{color}
\usepackage{amssymb}
\usepackage{graphicx}
\usepackage{amscd}

\usepackage[colorlinks=true,
linkcolor=webgreen,
filecolor=webbrown,
citecolor=webgreen]{hyperref}

\definecolor{webgreen}{rgb}{0,.5,0}
\definecolor{webbrown}{rgb}{.6,0,0}

\usepackage{color}
\usepackage{fullpage}
\usepackage{float}

\usepackage{graphics,amsmath,amssymb}
\usepackage{amsthm}
\usepackage{amsfonts}
\usepackage{latexsym}
\usepackage{epsf}

\usepackage{verbatim}

\usepackage{tikz}
\usepackage{forest}

\usepackage{mathtools}

\DeclarePairedDelimiter{\ceil}{\lceil}{\rceil}
\DeclarePairedDelimiter\floor{\lfloor}{\rfloor}

\theoremstyle{plain}
\newtheorem{theorem}{Theorem}
\newtheorem{corollary}[theorem]{Corollary}
\newtheorem{lemma}[theorem]{Lemma}

\theoremstyle{definition}

\theoremstyle{remark}

\begin{document}

\title{Confirming the Labels of Coins in One Weighing}
\author{Isha Agarwal \and Paul Braverman \and Patrick Chen \and William Du \and Kaylee Ji \and Akhil Kammila \and Tanya Khovanova \and Shane Lee \and Alicia Li \and Anish Mudide \and Jeffrey Shi \and Maya Smith \and Isabel Tu}

\maketitle

\begin{abstract}
There are $n$ bags with coins that look the same. Each bag has an infinite number of coins and all coins in the same bag weigh the same amount. Coins in different bags weigh 1, 2, 3, and so on to $n$ grams exactly. There is a unique label from the set 1 through $n$ attached to each bag that is supposed to correspond to the weight of the coins in that bag. The task is to confirm all the labels by using a balance scale once. 

We study weighings that we call downhill: they use the numbers of coins from the bags that are in a decreasing order. We show the importance of such weighings. We find the smallest possible total weight of coins in a downhill weighing that confirms the labels on the bags. We also find bounds on the smallest number of coins needed for such a weighing.
\end{abstract}

\section{Introduction}\label{sec:intro}

Coin puzzles have fascinated mathematicians for a long time. Guy and Nowakowsky summarized the most famous coin problem in their paper \cite{GN}. We are interested in the particular famous coin-weighing puzzle below. This puzzle appeared on the 2000 Streamline Olympiad for 8th grade for the case of $n=6$.

\begin{quote}
\textbf{Puzzle.} You have $n$ bags with coins that look the same. Each bag has an infinite number of coins and all coins in the same bag weigh the same amount. Coins in different bags weigh 1, 2, 3, 4, $\ldots$, $n-1$, and $n$ grams exactly. There is a label, which is an integer from 1 to $n$, attached to each bag that is supposed to correspond to the weight of the coins in that bag. You have only a balance scale. What is the least number of times that you need to weigh coins in order to confirm that the labels are correct? 
\end{quote}

The answer is 1 for all $n$, and we present an example of such a weighing. 

\textbf{Example.} On the right pan, place one coin from the bag labeled 2, two coins from the bag labeled 3, and so on such that there are $i-1$ coins from the bag labeled $i$.  On the left pan, put coins from the bag labeled 1 to match the right pan in total weight. For example, if $n = 4$, we put 20 coins labeled 1 on the left pan and coins with the labels 2, 3, 3, 4, 4, and 4 on the right pan. Any rearrangement of the bags makes the left pan heavier or the right pan lighter. Thus, all the labels are confirmed if the scale balances.

Our first goal in this paper is to find a weighing that minimizes the total weight on both pans. Our second goal is to minimize the number of coins used.

In Section~\ref{sec:prelim}, we provide definitions, examples, and basic results. We study downhill weighings, which are weighings that have decreasing multiplicities, that is weighings that use fewer coins of weight $i$ than weight $j$, where $i > j$. For consistency, we assume that the multiplicities of the coins on the right pan are negative. In Section~\ref{sec:seppoint}, we explain the idea of a separation point and how it helps to calculate the minimum weight for a downhill weighing. In Sections~\ref{sec:3k+1}, \ref{sec:3k}, and \ref{sec:3k+2}, we calculate the minimum weight explicitly depending on the remainder of $n$ modulo 3. Section~\ref{sec:3k+1} is devoted to the simplest case of $n = 3k+1$. In this case, the minimum weight is $\frac{8n^3+12n^2-12n-8}{81}$. In Section~\ref{sec:3k} we study the case of $3k$ and find the minimum weight to be $\frac{8n^3+ 27n^2 +9n-81}{81}$. In Section~\ref{sec:3k+2} we study the case of $3k+2$. This case is more complicated than other cases. We show that the minimum weight is $\frac{8n^3+56n^2-58n-10}{81}$ with eight exceptions. In all the cases, the minimum weight grows as $\frac{8n^3}{81}$. We present data for the sequences of our lower bound for the minimum weight and the actual minimum weight in Section~\ref{sec:data}.

Section~\ref{sec:mincoins} is devoted to finding the minimum number of coins. We show that for $n = 3k+1$ the answer is $\frac{5n^2 -n -4}{18}$. For other $n$ we find the lower and upper bounds on the number of coins. Both bounds are approximately $\frac{5n^2}{18}$.

Section~\ref{sec:nottight} discusses downhill imbalances with the weight difference of more than 1. Such imbalances use more coins and more weight. The last two sections discuss amusing examples of weighings. Section~\ref{sec:solo} is devoted to weighings where one side of the pan has coins of one type. Section~\ref{sec:ap} looks at weighings forming an arithmetic progression.

\section{Preliminaries}\label{sec:prelim}

\subsection{Some Definitions}

We call a weighing \textit{verifying} if it proves that all the labels on the bags are correct.

We call a verifying weighing \textit{coin-optimal} if it uses the minimum number of coins. We call a verifying weighing \textit{weight-optimal} if the total weight is minimal. Many properties we discuss later are true for both coin-optimal and weight-optimal weighings. So we call a weighing \textit{optimal} if it is either coin- or weight-optimal.

There should not be coins of the same type on both sides in an optimal weighing. Otherwise, we can remove both coins and have a solution with fewer total coins and less total weight. We call a weighing that does not have the same type of coins on both pans \textit{an economical weighing}. From now on we assume that all weighings are economical.

Let $a_i$ be the number of coins of type $i$ that is placed on the left pan minus the number of coins of type $i$ on the right pan. If the coin of type $i$ is not on any pan during a weighing, then $a_i=0$. We call numbers $a_i$ \textit{multiplicities}. It follows, that if coins of type $i$ are placed only on the left pan, then $a_i$ is positive and if they are placed only on the right pan, then $a_i$ is negative.

For example, if we compare one coin of weight 2 and two coins of weight 1 on the left pan against one coin of weight 4 on the right pan, then we write the set of multiplicities as $\{2,1,0,-1\}$. This weighing is verifying. Sometimes we can also write this weighing as $1 + 1 + 2 = 4$, or as $112=4$, where the equal sign denotes that this is a balanced weighing. We use $>$ and $<$ to denote imbalances. 

The total number of coins used in the weighing is $\Sigma_{i=1}^n |a_i|$. The total weight is $\Sigma_{i=1}^n |i \cdot a_i|$.

The set of multiplicities in a verifying weighing cannot contain duplicates. If so, then we could swap the coins corresponding to these multiplicities among themselves without changing the weighing result. Thus, we cannot distinguish the corresponding bags, so the weighing would no longer be verifying.

\begin{lemma}
If a weighing with a set of multiplicities $A$ is verifying, then a weighing with a set of multiplicities $kA$ is also verifying.
\end{lemma}

\begin{proof}
Suppose a weighing with the set of multiplicities $kA$ is not verifying. Hence, there exists
some rearrangement of coins such that the pans do not change their relative weights. Taking
this rearrangement of coins, dividing the weights of both pans by $k$ keeps the pans in the same position as the pans filled according to multiplicities $A$. Hence, $A$ is not verifying, creating a contradiction.
\end{proof}

That means, we only need to study the weighings with multiplicities $A$ such that $gcd(A) = 1$. We call such sets \textit{primitive}.

We say that the coins from the same bag are of the same \textit{type}. We need at least $n-1$ types of coins present in a verifying  weighing: if two types are not present, they cannot be distinguished between each other.

Without loss of generality, for unbalanced weighings, we assume that the left pan is lighter.

\subsection{Downhill weighings}

We call weighings with decreasing multiplicities \textit{downhill} weighings. These weighings play an important role in this paper.

\begin{lemma}\label{lemma:downhill}
If a verifying weighing is an imbalance (left side lighter), then the multiplicities are decreasing: $a_i > a_j$ for $i < j$. That is, it is a downhill weighing.
\end{lemma}

\begin{proof}
Suppose $a_i \leq a_j$ for $i < j$. Then $(j-i)(a_j-a_i) \geq 0$, which means $ia_i + ja_j \geq ia_j + ja_i$. That means, swapping bags $i$ and $j$ keeps the imbalance in place. Therefore, we can not differentiate between coins of type $i$ and $j$. Since this is a contradiction, an imbalanced verifying weighing must be downhill.
\end{proof}

\begin{corollary}
If a verifying weighing is an imbalance with one type, $x$, of coins missing from the scale, then the lighter pan contains coins that are lighter than $x$, and the heavier pan contains coins that are heavier than $x$.
\end{corollary}

\begin{proof}
If coins of type $x$ are not on the scale, then $a_x=0$. Therefore, for $i<x$, multiplicities $a_i$ are positive, which means that the corresponding coins are on the left side. A similar statement is true for the right side.
\end{proof}

We showed that a verifying imbalance must be downhill. The imbalances with a difference of weight equal to 1 play an important role in this paper. We call such imbalances \textit{tight}. 

\begin{lemma}\label{lemma:tightdownhillimbalance}
A tight downhill imbalance is verifying.
\end{lemma}

\begin{proof}
Any permutation of bags increases the left side, decreases the right side, or both. This changes the outcome of the weighing from the left pan being lighter to a balance or to a left pan being heavier. Therefore, this is a verifying weighing.
\end{proof}

\subsection{Balances}

Balances are more complicated than imbalances. A verifying balance does not have to be downhill. For example, with 3 bags, the weighing $1113 = 222$ with multiplicities $\{3,-3,1\}$ is verifying but not downhill. Another such example for 4 bags is $11114=332$ with multiplicities $\{4,-1,-2,1\}$. However, we can prove that all downhill balances are verifying. 

On another hand, a downhill balance is always verifying for the same reason a tight imbalance is verifying: any permutation of bags increases the left side or decreases the right side.

\begin{lemma}
A downhill balance is verifying.
\end{lemma}

\begin{proof}
This follows from the same logic used in Lemma~\ref{lemma:tightdownhillimbalance}.
\end{proof}

In our computational experiments, we found that in general, balances that are not downhill use more total coins and have greater total weight than downhill weighings. One exception is for 3 bags; the weighing $1+1=2$ with multiplicities $\{2,-1,0\}$ is weight-optimal.
	
For the rest of the paper, we assume that all weighings are \textbf{downhill}.

\subsection{Small number of bags}

We can directly find the downhill weighings that are coin- and weight-optimal for small values of the number of bags. The results are in Table~\ref{table:smallnumberofbags}.

\begin{table}[htb]
\centering
  \begin{tabular}{| c | c | c|c  |c|}
\hline
Number of bags  & Example & Multiplicities & \#coins & Total weight \\
\hline
2		& $1 < 2$ 		& $\{1,-1\}$		& 2		& 3\\
3		& $11 < 3$		& $\{2,0,-1\}$		& 3		& 5\\
4		& $112 = 4$		& $\{2,1,0,-1\}$	& 4		& 8\\
5		& $111223 = 55$ 	& $\{3,2,1,0,-2\}$	& 8		& 20 \\
6		& $111122233 < 566$	& $\{4,3,2,0,-1,-2\}$	& 12		& 33 \\
7		& $1111222334 = 677$	& $\{4,3,2,1,0,-1,-2\}$ & 13		& 40 \\
\hline
  \end{tabular}
  \caption{Minimum number of bags for downhill weighings for $n < 8$.}\label{table:smallnumberofbags}
\end{table}

\subsection{An imbalance that is not tight}

Consider an imbalance that is not tight.

\begin{lemma}
If, in a verifying imbalance, the difference of the weights on the two pans is $d$, then the difference between two consecutive multiplicities is greater than or equal to $d$. 
\end{lemma}

\begin{proof}
Suppose multiplicities are $a_1, a_2, a_3,\ldots, a_n$, where $a_1 > a_2 > a_3 > .... > a_n$ by Lemma~\ref{lemma:downhill}. The scales tilts depending on the sign of
\[a_1 + 2a_2 + 3a_3 + \cdots + na_n.\]
Suppose the difference $a_k -a_{k+1} < d$. After we switch the coins labeled $k$ and $k+1$, the new weighing tilts according to 
\[a_1 + 2a_2 + 3a_3 + \cdots ka_{k+1} + (k+1)a_k + \cdots + na_n.\]
The new expression differs from the previous one by 
\[ka_{k+1}  + (k+1)a_k - ka_{k}  + (k+1)a_{k+1} = a_{k+1} -a_k,\]
which is less than $d$. That means the scale tilts the same way as before. Hence, such a weighing is not verifying. It follows that the consecutive multiplicities need to differ by at least $d$.
\end{proof}

We see that imbalances that are not tight need more coins and have greater weight. In Section~\ref{sec:nottight}, we prove that imbalances which minimize the number of coins and the total weight must to be tight. To prove this, we first find bounds for coin- and weight-optimal balances and tight imbalances. Then we show that the weight and the number of coins in non-tight imbalances exceed these bounds.

\subsection{Useful formulae}

The following well-known formulae are used repeatedly throughout this paper.

In the first formula, we have $k$ products of two numbers such that one of the numbers is decreasing by one and the other is increasing by one:
\begin{multline}
ab + (a+1)(b-1) + \cdots + (a+k-1)(b-k+1)=\\
kab + \frac{k(k - 1)}{2}(b - a) - \frac{k(k - 1)(2k - 1)}{6}.
\end{multline}

One important case is $a=1$ and $b=k$:
\begin{multline}
1 \cdot k + 2(k-1) + \cdots + k\cdot 1=k^2 + \frac{k(k - 1)}{2}(k - 1) -\frac{k(k - 1)(2k - 1)}{6} = \\
\frac{k(k+1)(k+2)}{6} =\binom{k+2}{3}.
\end{multline}
This sequence as a function of $k$ is the sequence of tetrahedral numbers. It is sequence A000292 in the OEIS \cite{OEIS}.

In the next important formula, we have $k$ products of two numbers such that both numbers are increasing by 1:

\begin{multline}
ab + (a+1)(b+1) + \cdots + (a+k-1)(b+k-1)=\\
kab + \frac{k(k - 1)}{2}(b + a) + \frac{k(k - 1)(2k - 1)}{6}.
\end{multline}

An important case to consider is $a=1$:
\begin{multline}
b + 2(b+1) + \cdots + k(b+k-1)=kb + \frac{k(k - 1)}{2}(b + 1) + \frac{k(k - 1)(2k - 1)}{6}=\\
\frac{k(k + 1)}{2}b + \frac{k(k - 1)(2k +2)}{6}.
\end{multline}

\subsection{First Upper bound}\label{sec:introUB}

Our first naive example in Section~\ref{sec:intro} provides an upper bound for the smallest total weight and the smallest number of coins in a verifying weighing. Here is the calculation.

The right pan contains $\frac{n(n-1)}{2}$ coins for a total weight of:
\[1\cdot 2 + 2\cdot 3 + \cdots + (n-1)\cdot n = \frac{(n-1)n(n+1)}{3}.\]

The left pan contains coins of weight 1 to match the weight of the right pan. That means the total number of coins used is 
\[\frac{(n-1)n(n+1)}{3}+\frac{n(n-1)}{2},\]
and the total weight is
\[\frac{2(n-1)n(n+1)}{3}.\]
Approximately, we use $\frac{1}{3}n^3$ coins for a total weight of $\frac{2}{3}n^3$.

\section{Separation point and bounding weights}\label{sec:seppoint}

Let us define a \textit{separation point} $s$ to be the smallest label on the right pan in a downhill weighing. Given a separation point, the set of distinct multiplicities with minimal sum is $s - 2, s - 3, s - 4, \ldots, 1, 0, -1, -2, \cdots, -(n - s + 1)$ where $n$ is the total number of bags.

We denote the following quantity by $W_L(s,n)$:
\[W_L(s,n) = 1\cdot(s-2) + 2\cdot(s-3) + \cdots + (s-3)\cdot 2 + (s-2)\cdot 1 = \binom{s}{3}.\]

We call $W_L(s,n)$ the \textit{the bounding-left weight} as the weight of coins on the left pan in any downhill weighing with $n$ bags and separation point $s$  is at least $W_L(s,n)$.

Similarly we call $W_R(s,n)$ the \textit{bounding-right weight}, where
\[W_R(s,n) = 1\cdot s + 2\cdot(s+1) + \cdots + (n-s+1)\cdot n = \frac{(s - n - 2)(s - n - 1)(s + 2n)}{6}.\]

Suppose $W_L(s,n) \geq W_R(s,n)$. In a downhill weighing the right pan has to weigh at least the same as the left. Hence, the total weight in an optimal downhill verifying weighing with the separation point $s$ has to be at least $2W_L(s,n)$.

Now, suppose $W_L(s,n) < W_R(s,n)$. 

If the downhill weighing is tight, then the total weight in a verifying weighing has to be at least $2W_R(s,n) - 1$. 
We discuss the case of the imbalance that is not tight in Section~\ref{sec:nottight}.

Let us denote $W_B(s,n)$ as the value $2W_R(s,n) - 1$ if $W_L(s,) < W_R(s,n)$ and $2W_L(s,n)$ otherwise. We call the integer $W_B(s,n)$ the \textit{bounding weight}. As demonstrated above, a tight downhill verifying weighing with a given separation point has to have a total weight of at least the bounding weight.

We denote the minimum bounding weight as $W_B(n)$, i.e.\ $W_B(n) = \min_s\{W_B(s,n)\}$. We denote the minimum possible weight in a downhill weighing as $W_M(n)$.

Bounding weights are functions of the number of bags and a separation point. We call $m$ the \textit{the minimal separation point} if the corresponding bounding weight is minimal, i.e.\ $W_B(n) = W_B(m,n)$. We call the corresponding bounding weight $W_B(m,n)$ the minimal bounding weight. The lemma below follows.

\begin{lemma}
The minimum weight is at least the minimal bounding weight: $W_M(n) \geq W_B(n)$. If the minimal bounding weight can be achieved in a weighing, then it is the minimum weight.
\end{lemma}

The polynomial $W_L(s,n) = \binom{s}{3}$ increases when $s$ ranges from 1 to $n$, while $W_R(s) = \frac{(s - n - 2)(s - n - 1)(s + 2n)}{6}$ decreases. They equal to each other for $s = \frac{2n + 4}{3}$. Unfortunately, this value is not always an integer. Thus, we divide our further investigations into three cases depending on the remainder of $n$ modulo 3.

In the next three sections, we find the minimal weight for $n = 3k + 1$, $n = 3k$, and $n = 3k + 2$, where k is a non-negative integer.

\section{Case $n = 3k + 1$. Minimum weight.}\label{sec:3k+1}

\begin{lemma}
If $n = 3k + 1$, the minimum weight in a downhill weighing, which is balanced or tight imbalanced, is $\frac{8n^3+12n^2-12n-8}{81}$. Furthermore, it is achievable as a balance.
\end{lemma}

\begin{proof}
If $n = 3k + 1$, then our minimum separation point, $s=\frac{2n + 4}{3} = 2k+2$, is an integer. For this separation point 
\[W_L(2k+2,3k+1)= \binom{2k+2}{3} = \frac{2k(2k+1)(2k+2)}{6},\]
and
\begin{equation*}
\begin{split}
W_R(2k+2,3k+1) & = \frac{(2k+2 - 3k-1-2)(2k+2 - 3k-1-1)(2k+2+6k+2)}{6}\\
& = \frac{(-k-1)(-k)(8k+4)}{6}.
\end{split}
\end{equation*}

As $W_L(s,n)=W_R(s,n)$, for $n = 3k + 1$, the minimum bounding weight is $\frac{2k(2k+1)(2k+2)}{3}$. It is achievable, and therefore, the minimum weight is:
\[W_M(3k+1) = W_B(3k+1) = 2W_L(2k+2,3k+1)= \frac{2k(2k+1)(2k+2)}{3}=\frac{8n^3+12n^2-12n-8}{81}.\]
\end{proof}

For $3k+1$ bags, the minimum weight as a function of $k >0$ starts as: 8, 40, 112, 240, 440, and so on. These numbers are always even as they are equal to $2W_L(m,n)$. This sequence is not in the OEIS \cite{OEIS}, but the  sequence $W_L(2k+2,3k+1)$ is there. It is sequence A002492, which is the sum of the first $n$ even squares: 4, 20, 56, 120, 220, and so on.

\section{Case $n = 3k$. Minimum weight.}\label{sec:3k}

\begin{lemma}
If $n = 3k$, the minimum weight in a downhill weighing, which is balanced or tight imbalanced, is $\frac{8n^3+ 27n^2 +9n-81}{81}$. It is achievable as a tight imbalance.
\end{lemma}

For $n \neq 3k+1$, the expression $\frac{2n + 4}{3}$ is not an integer. Thus, the minimal separation point is either $\floor{\frac{2n + 4}{3}}$ or $\ceil{\frac{2n + 4}{3}}$. In our case of $n =3k$, the minimal separation point is either $2k+1$ or $2k+2$. If the separation point is $2k+1$, then $W_L(2k+1,3k) = \frac{8k^3-2k}{6}$ and $W_R(2k+1,3k)=\frac{8k^3+ 9k^2 +k}{6}$. In this case $W_R(2k+1,3k) > W_L(2k+1,3k)$, and the bounding weight is 
\[W_B(2k+1,3k) = 2W_R(2k+1,3k)-1 = \frac{8k^3+ 9k^2 +k}{6} -1,\]
which is achievable. Indeed, we can add any weight to the left side and still have a downhill weighing by using coins of type 1.

If the separation point is $2k+2$, then $W_R(s,n) < W_L(s,n)= \frac{8k^3+12k^2+4k}{6}$ and the bounding weight is
\[W_B(2k+2,3k) = 2W_L(2k+2,3k) = \frac{8k^3+12k^2+4k}{3}.\]
As $W_B(2k+1,3k) < W_B(2k+2,3k)$, the minimal separation point is $2k+1$. The minimal bounding weight is achievable. Thus, the minimum weight is: 
\[W_M(3k) =W_B(3k) = 2W_R(2k+1,3k)-1= \frac{8k^3+ 9k^2 +k}{3}-1 = \frac{8n^3+ 27n^2 +9n-81}{81}.\]

For $3k$ bags, the minimum weight as a function of $k >0$ starts as: 5, 33, 99, 219, and so on. These numbers are always odd, as they are equal to $2W_R(m,n)-1$. This sequence is not in the OEIS \cite{OEIS}, but the sequence corresponding to $W_R(m,n)$ is there. It is sequence A132124$(n) =\frac{n(n + 1)(8n + 1)}{6}$, which starts as 0, 3, 17, 50, 110, 205, 343, 532, 780, 1095, 1485, 1958, and so on.

\section{Case $n = 3k+2$. Minimum weight.}\label{sec:3k+2}

In the case of $n=3k+2$, the minimal separation point is either $2k+2$ or $2k+3$. If the separation point is $2k+2$, then $W_R(2k+2,3k+2) > W_L(2k+2,3k+2)$, and the bounding weight is 
\[W_B(2k+2,3k+2) = 2W_R(2k+2,3k+2)-1 =\frac{8k^3+30k^2+34k+6}{3}-1.\]

If the separation point is $2k+3$, then $W_R(2k+2,3k+2) < W_L(2k+3,3k+2)$, and the bounding weight is 
\[W_B(2k+2,3k+2) = 2W_L(2k+3,3k+2) =\frac{8k^3+24k^2+22k+6}{3}.\]

The point $2k+3$ is the minimal separation point with the bounding weight
\[W_B(3k+2)= W_B(2k+3,3k+2)=\frac{8k^3+24k^2+22k+6}{3}.\] 

Now we want to see if this weight is achievable, so we calculate the right weight:
\[
\begin{split}
W_R(2k+3,3k+2) = \frac{(2k+3 - 3k-2 - 2)(2k+3 - 3k-2 - 1)(2k+3 + 6k+4)}{6}\\
= \frac{8k^3+ 15k^2 +7k}{6}.
\end{split}
\]

Consider the difference $W_L(2k+3)- W_R(2k+3) =\frac{3k^2 +5k+2}{2}= \frac{(3k+2)(k+1)}{2}$. Now we have to add some coins to the right pan so that the amount added is at least $\frac{(3k+2)(k+1)}{2}$, and the weighing is downhill. It is not always possible to add some more coins to the right pan so that their weight sums up exactly to $\frac{(3k+2)(k+1)}{2}$. We consider cases separately when $k$ is odd or even.

We later show in this section that the minimum bounding weight is achievable with eight exceptions. 

For $3k+2$ bags, the minimum \textbf{bounding} weight as a function of $k >0$ starts as: 20, 70, 168, 330, 572. This is sequence A259110 in the OEIS \cite{OEIS}. The numbers are always even as they are equal to $2W_L$. The  sequence $W_L$ is also in the OEIS. It is sequence A000447: sum of odd squares.

\subsection{Subcase $n = 3k+2$, $k$ odd.} If $k=2j+1$ for some integer $j$, we can add the largest weighing coin $\frac{k+1}{2}$ times to the right pan. In this case, the bounding weight is achievable and equal to
\[W_M(6j+5) = W_B(6j+5) =\frac{8k^3+24k^2+22k+6}{3}=\frac{8n^3+56n^2-58n-10}{81}.\] 

For $3k+2$, where $k=2j+1$, the minimum weight as a function of $j \geq 0$ starts as: 20, 168, 572, 1360, and so on. This sequence is not in the OEIS \cite{OEIS}, but the sequence corresponding to $W_L(m,n)$ is there. It is sequence A267031$(n) = (32n^3 - 2n)/3$, which starts as 10, 84, 286, 680, 1330, 2300.

Thus, we have proved the following lemma.

\begin{lemma}
If $n = 3k+2$, where $k$ is odd, the minimum weight in a downhill weighing, which is balanced or tight imbalanced, is $\frac{8n^3+56n^2-58n-10}{81}$. Furthermore, it is achievable as a balance.
\end{lemma}

\subsection{Subcase $n = 3k+2$, $k$ even.}

Before discussing this case, we must first introduce the notation for triangular numbers. Suppose $T_m$ is the $m$-th triangular number, that is 
\[T_m = 1 + 2 + \cdots + m = \frac{m(m+1)}{2}.\]

\begin{lemma}
For $n = 3k+2$, where $k$ is even, and $n > 50$, the minimum weight in a downhill weighing, which is either tight imbalanced or balanced, is $\frac{8n^3+56n^2-58n-10}{81}$. Furthermore, it is achievable as a balance.
\end{lemma}

\begin{proof}
Our separation point is $s = 2k + 3$, and the difference between the RHS (right-hand side) and LHS (left-hand side) is $\frac{3k^2 + 5k +2}{2} = n\cdot \frac{n+4}{6} - \frac{n}{2}$.

Consider the coins added to the RHS in order to make it greater than or equal to the LHS and result in a downhill weighing. Let the multiplicity of the additional coins weighing $i$ be $a_i$, where $s \leq i \leq n$. The total added weight is:
$$a_nn + a_{n-1}(n-1) + a_{n-2}(n-2) + \cdots + a_{s}s,$$
where $a_n \geq a_2 \geq a_{n-1} \geq \ldots \geq a_s$ to keep the weighing downhill.

We define non-negative integers $b_1, b_2, b_3, \ldots, b_{n-s+1}$ such that $b_{n-s+1} = a_{s}$ and $b_{i} = a_{n-i} - a_{n-i-1}$ for $1 \leq i < n-s+1$.

With this notation we can represent the total weight $a_1n + a_2(n-1) + a_3(n-2) + \cdots + a_{n-s+1} \cdot s$ as $b_1{n} + b_2(2n-1) + b_3(3n-3) + \cdots + b_{n-s+1}(n(n-s+1) - T_{n-s})$. Note that the largest triangular number in the above expression is $T_{n-s} = T_{n-2k-3} = T_{k-1}=T_{\frac{n-2}{3}-1}$.

This sum can be expressed as a multiple of $n$ minus a sum of triangular numbers:
\[(b_1 + 2b_2 + 3b_3 + \cdots b_{n-s+1}(n-s+1))n - b_2T_1 - b_3T_2 - \cdots - b_{n-s+1}T_{n-s}.\]
Additionally, we have
\[(b_1 + 2b_2 + 3b_3 + \cdots b_{n-s+1}(n-s+1) = \frac{n+4}{6}\]
and 
\begin{equation}\label{eq:t}
 b_2T_1 + b_3T_2 + \cdots + b_{n-s+1}T_{n-s} = \frac{n}{2}.
\end{equation}

By reversing the procedure above, we see that if we can express $\frac{n}{2}$ as a sum of triangular numbers $T_i$, where $i < \frac{n-2}{3}$, then we can get the multiplicities $a_i$ and, correspondingly, the set of coins to add to the right pan to achieve the minimum weight.

Suppose we find an expression for $n=z$ using triangular numbers as in (Eq.)~\ref{eq:t}. Then, the expression works for $n= z+42$. We increase $b_7$ by 1. Thus, the left side in (Eq.)~\ref{eq:t} increases by 21, which is matched by increasing $n$ by 42. It follows that if the minimum bounding weight can be achieved for $n$, then it can be achieved for $n+42$. By computer search, we found solutions for $50 < n < 10000$. It follows that there are solutions for any $n > 50$.
\end{proof}

For $3k+2$, where $k = 2j$, the minimum weight as a function of $j >0$ starts as: 70, 330, 910, 1938, 3542, and so on. This sequence is not in the OEIS \cite{OEIS}, but the sequence corresponding to $W_L(m,n)$ is there. It is sequence A015219 of odd tetrahedral numbers.

\subsubsection{Subcase $n = 3k+2$, $k$ even. Small values}

We showed that for $n > 50$ the bounding weight is achievable. We found the minimum weight for smaller values of $n$ by an exhaustive search. The results are represented in Table~\ref{table:Minimalweight}. The last column shows the number of coins used in our weight-optimal weighings.

\begin{table}[htb]
\centering
  \begin{tabular}{| c |c| c|c|c|}
\hline
\#bags  & Min Weight & Min Bounding Weight & Difference & \#coins\\
\hline
8	& 75 & 70 & 5 & 22 \\
14	& 337 & 330 & 7 & 60 \\
20	& 917 & 910 & 7 & 118 \\
26	& 1943 & 1938 & 5 & 196 \\
32	& 3543 & 3542 & 1 & 292 \\
38	& 5857 & 5850 & 7 & 412 \\
44	& 8991 & 8990 & 1 & 548 \\
50	& 13095 & 13090 & 5 & 708 \\
\hline
  \end{tabular}
  \caption{Minimal weight for $n = 3k+2$, $k$ even, $n < 51$.}\label{table:Minimalweight}
\end{table}

The sequence of differences for minimum bounding weight is: 5, 7, 7, 5, 1, 7, 1, 5, 0, 0, 0, 0, 0, 0, $\dots$.

Here we provide the multiplicities for the minimum weight for the exceptional cases.

8 bags: \{7, 4, 3, 2, 1, 0, $-$2, $-$3\}

14 bags: \{10, 9, 7, 6, 5, 4, 3, 2, 1, 0, $-$1, $-$3, $-$4, $-$5\}

20 bags: \{14, 13, 11, 10, 9, 8, 7, 6, 5, 4, 3, 2, 1, 0, $-$1, $-$2, $-$4, $-$5, $-$6, $-$7\}

26 bags: \{19, 16, 15, 14, 13, 12, 11, 10, 9, 8, 7, 6, 5, 4, 3, 2, 1, 0, $-$1, $-$2, $-$3, $-$5, $-$6, $-$7, $-$8, $-$9\}

32 bags: \{21, 20, 19, 18, 17, 16, 15, 14, 13, 12, 11, 10, 9, 8, 7, 6, 5, 4, 3, 2, 1, 0, $-$1, $-$2, $-$3, $-$4, $-$6, $-$7, $-$8, $-$9, $-$10, $-$11\}

38 bags: \{26, 25, 23, 22, 21, 20, 19, 18, 17, 16, 15, 14, 13, 12, 11, 10, 9, 8, 7, 6, 5, 4, 3, 2, 1, 0, $-$1, $-$2, $-$3, $-$4, $-$5, $-$6, $-$8, $-$9, $-$10, $-$11, $-$12, $-$14\}

44 bags: \{29, 28, 27, 26, 25, 24, 23, 22, 21, 20, 19, 18, 17, 16, 15, 14, 13, 12, 11, 10, 9, 8, 7, 6, 5, 4, 3, 2, 1, 0, $-$1, $-$2, $-$3, $-$4, $-$5, $-$6, $-$7, $-$9, $-$10, $-$11, $-$12, $-$13, $-$14, $-$16\}

50 bags: \{35, 32, 31, 30, 29, 28, 27, 26, 25, 24, 23, 22, 21, 20, 19, 18, 17, 16, 15, 14, 13, 12, 11, 10, 9, 8, 7, 6, 5, 4, 3, 2, 1, 0, $-$1, $-$2, $-$3, $-$4, $-$5, $-$6, $-$7, $-$8, $-$9, $-$11, $-$12, $-$13, $-$14, $-$15, $-$17, $-$18\}

\section{Minimum Weight Data}\label{sec:data}

We combined the formulae for the minimum weight in Table~\ref{table:mw}. Recall that the case $n= 3k+2$ has exceptions when $n <50$.

\begin{table}[htb]
\centering
  \begin{tabular}{| c |c|}
\hline
\#bags  & Min Weight \\
\hline
$3k+1$	& $\frac{8n^3+12n^2-12n-8}{81}$ \\
$3k$	& $\frac{8n^3+ 27n^2 +9n-81}{81}$ \\
$3k+2$	& $\frac{8n^3+56n^2-58n-10}{81}$ \\
\hline
  \end{tabular}
  \caption{Minimal weight.}\label{table:mw}
\end{table}

In any case, the minimum weight is not more than 
\[\frac{8n^3+56n^2+9n-8}{81}.\]

The sequence of minimal bounding weights $W_B(n)$, which starts with three bags, is: 5, 8, 20, 33, 40, 70, 99, 112, 168, 219, 240, 330, 409, 440, 572, 685, 728, 910, 1063, 1120, 1360, 1559, 1632, 1938, 2189, 2280, 2660, 2969, 3080, 3542, 3915, 4048, 4600, 5043, 5200, 5850, 6369, 6552, 7308, 7909, 8120, 8990, 9679, 9920, 10912, 11695, 11968, 13090, 13973, 14280, 15540, 16529, 16872, 18278, 19379, 19760, 21320, 22539, 22960, 24682, 26025, 26488, 28380, 29853, 30360, 32430, 34039, 34592, 36848, 38599, 39200, 41650, 43549, 44200, 46852, 48905, 49608, 52470, 54683, 55440, 58520, 60899, 61712, 65018, 67569, 68440, 71980, 74709, 75640, 79422, 82335, 83328, 87360, 90463, 91520, 95810, 99109, 100232, and so on.

The sequence of minimal weights $W_M(n)$ differs from the previous sequence in eight places and is: 5, 8, 20, 33, 40, \textbf{75}, 99, 112, 168, 219, 240, \textbf{337}, 409, 440, 572, 685, 728, \textbf{917}, 1063, 1120, 1360, 1559, 1632, \textbf{1943}, 2189, 2280, 2660, 2969, 3080, \textbf{3543}, 3915, 4048, 4600, 5043, 5200, \textbf{5857}, 6369, 6552, 7308, 7909, 8120, \textbf{8991}, 9679, 9920, 10912, 11695, 11968, \textbf{13095}, 13973, 14280, 15540, 16529, 16872, 18278, 19379, 19760, 21320, 22539, 22960, 24682, 26025, 26488, 28380, 29853, 30360, 32430, 34039, 34592, 36848, 38599, 39200, 41650, 43549, 44200, 46852, 48905, 49608, 52470, 54683, 55440, 58520, 60899, 61712, 65018, 67569, 68440, 71980, 74709, 75640, 79422, 82335, 83328, 87360, 90463, 91520, 95810, 99109, 100232, and so on. We highlighted the places where the two sequences differ.

\section{Minimizing the Number of Coins}\label{sec:mincoins}

\subsection{Minimizing Weight Versus Minimizing Coins}

Weight-optimal and coin-optimal weighings are not always the same. Below is an example of when more coins produce less total weight.

Consider the case where we have 9 bags. Using the optimal separation point of 7, we get the starting multiplicities $\{5, 4, 3, 2, 1, 0, -1, -2, -3\}$. Here, the LHS weighs 35 and the RHS weighs 50. Hence, we can add 14 to the LHS to optimize the total weight to 99. The corresponding multiplicities for the smallest number of coins are $\{8, 6, 4, 3, 1, 0, -1, -2, -3\}$ for a total of 28 coins.

However, adding 15 to the LHS still keeps the weighing verifying, and we can do this with multiplicities $\{6, 5, 4, 3, 2, 0, -1, -2, -3\}$ for a total of 26 coins. Notice that the second example uses fewer coins while having a greater total weight.

Still, the number of coins used in the weight-optimal weighing is close to the minimum, so we will calculate this number in this section.

\subsection{Estimate for the number of coins in our strategy for the minimum weight}

We start by calculating the number of coins used in weight-optimal weighings, which were described in the previous sections. We do this by analyzing individual cases. Note that the upper bound we find in this section is approximately $5n^2/18$, which is a big improvement from our initial bound of $n^3/3$ in Section~\ref{sec:introUB}. 

We can estimate the lower bound on the number of coins using the fact that the multiplicities must be different on each pan. As before, we call integer $s$ a \textit{separation point} if $s$ is the smallest label on the right pan in a downhill weighing. If the separation point is $s$, then the multiplicities on the left pan have to be at least $\{0, 1, 2, \ldots, s-2\}$ for a total of $\frac{(s-2)(s-1)}{2}$ coins, which we denote as $F_L(s,n)$. The multiplicities on the right have to be at least $\{1, 2, \ldots, n-s, n-s+1\}$ for a total of $\frac{(n-s+2)(n-s+1)}{2}$ coins, which we denote as $F_R(s,n)$. We denote $F_L(s,n) + F_R(s,n)$ as $F(s,n)$. Thus, we have proven the following lemma.

\begin{lemma}\label{lemma:multiplicitiesbound}
If the total number of bags is $n$, and the separation point is $s$ in a downhill weighing, we need at least $F(s,n) = F_L(s,n)+F_R(s,n)$ coins, which is equal to
\[\left(s - \frac{n+3}{2}\right)^2 + \frac{n^2-1}{4}.\]
\end{lemma}

\subsection{Case $n = 3k+1$.} 

If $n = 3k+1$, then the separation point is $2k+2$. Then the left pan has $k(2k+1)$ coins, and on the right pan has $\frac{k(k+1)}{2}$ coins. Thus, the total number of coins used is
\[\frac{k(5k+3)}{2}=\frac{(n-1)(5(n-1)+9)}{18}=\frac{(n-1)(5n+4)}{18}= \frac{5n^2-n-4}{18}.\]

The first few numbers that this bound produces are: 4, 13, 27, 46, 70, 99. They form sequence A147875 in the OEIS \cite{OEIS}, which is the second heptagonal numbers.

\subsection{Case $n = 3k$.} 

For $n = 3k$, the minimal separation point is $2k+1$. It follows that $W_L(2k+1,3k) = \frac{8k^3-2k}{6}$ and $W_R(2k+1,3k)=\frac{8k^3+ 9k^2 +k}{6}$. In this case, $W_R(s,n) > W_L(s,n)$, so the bounding weight is $2W_R(s,n)-1$, which is achievable. We can add the required weight by using the coins from bag 1. However, we can do it by using fewer coins. 

To estimate the number of coins, we start with $(s - \frac{n+3}{2})^2 + \frac{n^2-1}{4}$ from Lemma~\ref{lemma:multiplicitiesbound}, which is the number of coins used in the weighing before we add coins to balance the left and the right pan. This is equal to $\frac{5n^2-3n}{18}$. Now we need to add some number of coins to the left pan to get to the weight $W_R(s,n)-1$. That is, we need to add the weight $D = W_R(s,s)-1 - W_L(s,s)$ which is equal to

\[D =\frac{8k^3+ 9k^2 +k}{6}- 1 - \frac{8k^3-2k}{6} = \frac{(3k-2)(k+1)}{2} = \frac{(n-2)(n+3)}{6}.\]

To keep this weighing downhill, we add coins in groups of $1+2+\cdots +j$. Each group weighs a triangular number $T_j$. The largest triangular number available is 
\[T_{2k-1} = k(2k-1) = \frac{n}{3}\left(\frac{2n}{3}-1\right)=\frac{2n^2-3n}{9}.\]
Our weight $D = \frac{(n-2)(n+3)}{6}$ that needs to be added is no more than $T_{2k-1}$ for $n \geq 6$. It is well-known fact, since Gauss discovered it in 1796, that each integer can be written as a sum of three triangular numbers, possibly including zero. Thus, we can represent $D$ as $T_a+T_b+T_c$. We want to find an upper bound on $a+b+c$, which is the corresponding number of coins. We have
\[a(a+1) + b(b+1) +c(c+1) = 2D.\]
Equivalently,
\[(a+0.5)^2 + (b+0.5)^2 + (c+0.5)^2 = 2D + 0.75.\]
By Cauchy–Schwarz inequality, if
\[x^2 + y^2 +z^2 = R,\]
then 
\[x + y +z \leq \sqrt{3R}.\]
Thus, 
\[a+b+c \leq \sqrt{6D + 2.25} - 1.5 = \sqrt{n^2 +n - 3.75} - 1.5 \leq n+ 0.5 - 1.5 = n-1.\]

Thus, there is a way to add the weight to the left pan that uses no more than $n-1$ coins. The total number of coins is no more than:
\[\frac{5n^2-3n}{18}+n-1 = \frac{5n^2+15n-18}{18}.\]

The first few numbers that this bound produces are 4, 14, 29, 49, 74, 104, 139. This sequence is not in the OEIS.

\subsection{Case $n = 3k+2$ and $k$ is odd.} 

For $n = 3k+2$ where $k$ is odd, the minimal separation point is $2k+3$. To calculate the number of coins, we start with $(s - \frac{n+3}{2})^2 + \frac{n^2-1}{4}$ from Lemma~\ref{lemma:multiplicitiesbound}, which equals $\frac{5n^2+n-4}{18}$. We also need to add some number of coins on the right pan. As we showed before, we add $\frac{k+1}{2}$ or $\frac{n+1}{6}$. The total is
\[\frac{5n^2+4n-1}{18}.\] In this case, this is exactly the number of coins used in the corresponding weight-optimal weighing.

The first few numbers that this calculation produces are: 8, 36, 84, 152, 240, 348, 476, 624, 792. This sequence is not in the OEIS.

\subsection{Case $n = 3k+2$ and $k$ is even and $n > 50$.} 

For $n = 3k+2$ where $k$ is even, the minimal separation point is $2k+3$. To calculate the number of coins, we start with $(s - \frac{n+3}{2})^2 + \frac{n^2-1}{4}$ from Lemma~\ref{lemma:multiplicitiesbound}, which equals $\frac{5n^2+n-4}{18}$. The algorithm we used above adds exactly $\lfloor \frac{k+1}{2}\rfloor = \frac{k+2}{2}$ coins. The latter expression equals $\frac{n+4}{6}$. Thus, the total number of coins is $\frac{5n^2+n-4}{18} + \frac{n+4}{6}$, which equals
\[\frac{5n^2+4n+8}{18}.\]

The first few numbers that this bound produces are: 2, 20, 58, 116, 194, 292, 410, 548, 706, 884, 1082, 1300, 1538, 1796, 2074, 2372, 2690. These numbers divided by 2 produce sequence A079273 in the OEIS, which is the Octo numbers.
The weighing algorithm that we use in this description is valid for $n > 50$. Thus, the calculation for the number of coins is exact starting from 884. For smaller values we found weight-optimal weighings using 2, 20, 60, 118, 196, 292, 412, 548, 708 coins, and this sequence is not in the OEIS.

\subsection{Lower Bound for the Number of Coins}

As we showed in Lemma~\ref{lemma:multiplicitiesbound}, if the total number of bags is $n$, and the separation point is $s$ in a downhill weighing, we need at least
\[\left(s - \frac{n+3}{2}\right)^2 + \frac{n^2-1}{4}\]
coins. This value is minimized when $s = \frac{n+3}{2}$. This is equivalent to saying that to minimize the total number of coins, the zero multiplicity should be in the middle. This provides our first lower bound of $\frac{n^2-1}{4}$ coins. We want to improve this bound by considering weight.

We consider a separation point $s' = \lfloor\frac{2n+4}{3}\rfloor$ to be the largest separation point such that the left minimal weight does not exceed the right minimal weight. First, we need a preliminary lemma.

\begin{lemma}\label{lemma:movingsptoleft}
If $s$ is a separation point and $s < k$, then to have the weight $\binom{k}{3}$ on the left pan with the separation point $s$, we need more than $\frac{(k-2)(k-1)}{2}$ coins on the left pan.
\end{lemma}

\begin{proof}
The weight $\binom{k}{3}$ can be represented as a sum
\[1\cdot(k-2) + 2\cdot(k-3) + \cdots + (k-3) \cdot 2 + (k-2)\cdot 1.\]
This uses $\frac{(k-2)(k-1)}{2}$ coins. 
However, this expression for the LHS currently has separation point $k$. In order to have the same weight with separation point $s$ where $s < k$, we must remove some coins that weigh more than $s-1$ and rearrange their weight with coins that weigh less than $s$ while keeping the weighing downhill. As each removed coin weighs more than each added coin, the number of removed coins will be no more than the number of added coins. This means that we need at least $\frac{(k-2)(k-1)}{2}$ coins on the left pan.
\end{proof}

\begin{theorem}
In a coin-optimal weighing with $n$ bags, we need at least $F(s')$ coins.
\end{theorem}

\begin{proof}
The number of coins needed for the separation point $j < s'$ on the left pan is at least $F_L(s',n)$. If $j < s'$, then $F_R(j,n) > F_R(s',n)$. Therefore, the total number of coins needed for $j < s'$ is more than $F(s')$. We also know that, for $j > s'$, we have $F(j,n) > F(s',n)$. That means we need at least $F(s')$ coins.
\end{proof}

The calculation for $F(s',n)$ is presented in the following Table~\ref{table:Minimalseparationpointcoins}. The last two columns show $F(s',n)$ expressed in terms of $k$ and $n$.

\begin{table}[htb]
\centering
  \begin{tabular}{| c |c| c|c|c|}
\hline
Number of bags  & $s'$ & $s''$& $F(s')$ & $F(s')$\\
\hline
$n=3k$		& $2k+1$ & $2k+2$ & $\frac{5k^2- k}{2}$ & $\frac{5n^2 -3n}{18}$\\
$n=3k+1$	& $2k+2$ & $2k+2$ & $\frac{5k^2+3k}{2}$ & $\frac{5n^2 -n -4}{18}$\\
$n=3k+2$	& $2k+2$ & $2k+3$ & $\frac{5k^2+5k + 2}{2}$ & $\frac{5n^2-5n+8}{18}$\\
\hline
  \end{tabular}
  \caption{Minimal bounding number of coins}\label{table:Minimalseparationpointcoins}
\end{table}

As we can see, in any case our lower bound is approximately $\frac{5n^2}{18}$, which is the same growth as the upper bound.

\subsection{Summary}

We merge the lower and the upper bounds into Table~\ref{table:mergeLU}.

\begin{table}[htb]
\centering
  \begin{tabular}{| c |c| c|}
\hline
Number of bags  & Lower bound & Upper bound \\
\hline
$n=3k$			& $\frac{5n^2 -3n}{18}$		& $\frac{5n^2 + 15n - 18}{18}$	\\
$n=3k+1$		& $\frac{5n^2 -n -4}{18}$	& $\frac{5n^2 -n -4}{18}$ \\
$n=3k+2$, $k$ is odd	& $\frac{5n^2-5n+8}{18}$	& $\frac{5n^2 +4n -1}{18}$ \\
$n=3k+2$, $n > 50$, $k$ is even	& $\frac{5n^2-5n+8}{18}$	& $\frac{5n^2 +4n +8}{18}$ \\
\hline
  \end{tabular}
  \caption{Lower and Upper bound}\label{table:mergeLU}
\end{table}

We see that, for $n=3k+1$, both bounds are the same.

\begin{corollary}
For $n=3k+1$, the minimum number of coins is $\frac{5n^2 -n -4}{18}$.
\end{corollary}

The following Table~\ref{table:coinResults} shows our computational results. In the middle row we present the smallest number of coins that we found. The numbers in bold are proven to be the best bound by an exhaustive search or by the corollary above.

\begin{table}[htb]
\centering
  \begin{tabular}{| c |c| c|c|c| c| c|c|c|c|c|c|c|c|c|}
\hline
number of bags 	& 2 	& 3 & 4 & 5 & 6 & 7 & 8 & 9 & 10 & 11 & 12 & 13 & 14 & 15\\
upper bound 	& N/A 	& 4 & 4 &  8 &  14 & 13 & N/A & 29 & 27 & 36 & 49 & 46 & N/A & 74\\
best found	& \textbf{2}	& \textbf{3} & \textbf{4} & \textbf{8} & \textbf{12} & \textbf{13} & 22 & 26 & \textbf{27} & 36 & 47 & \textbf{46} & 60 & 70 \\
lower bound	& 1	& 2 & 4  &  6 &  9 & 13 & 16 & 21 & 27 & 31 & 38 & 46 & 51 & 60\\
\hline
  \end{tabular}
  \caption{Lower and Upper bound}\label{table:coinResults}
\end{table}

\section{Downhill weighings that are not tight}\label{sec:nottight}

Suppose a downhill verifying weighing is an imbalance that is not tight. As before, a separation point $s$ is the smallest weight used on the right pan. Then, the smallest set of multiplicities that is used on both pans is:
\[\{2(s-2),2(s-3),2(s-4),\ldots, 4,2,0,-1, -3, -5, \ldots,-2n+2s-1\}.\]
We want to calculate the number of coins used in this weighing as well as the weights on the right and the left pans. 

\begin{theorem}
The smallest number of coins cannot be achieved in a downhill imbalance that is not tight.
\end{theorem}

\begin{proof}
We can use the previous calculations by noticing that 
\begin{multline*}
\{2(s-2),2(s-3),2(s-4),\ldots, 4,2,0,-1, -3, -5, \ldots,-2n+2s-1\} = \\
2\{s-2,s-3,\ldots,2,1,0,-1,-2,\ldots, -n+s-1\} + \{0,\ldots, 0,0,0,1,1,\ldots,1\}.
\end{multline*}
The number of coins used is $2(s - \frac{n+3}{2})^2 + \frac{n^2-1}{2} - (n-s+1)$. The minimum is reached for $s=\frac{2n+5}{4}$, and the corresponding number of coins is $ \frac{4n^2-4n-1}{8}$. Asymptotically, this function grows faster than our worst bound of $\frac{5n^2 + 15n - 18}{18}$. In fact, it is bigger than our bound for $n > 6$. For $n \leq 6$, we calculated the exact minimum number of coins. This means that the smallest number of coins cannot be achieved in a weighing that is not tight.
\end{proof}

Now we direct our attention to the weight. 

\begin{theorem}
The smallest weight cannot be achieved in a downhill imbalance that is not tight.
\end{theorem}

\begin{proof}
Let us denote the weight of the coins with multiplicities above as $V(s,n)$. We see that, for $s < n$, the difference of the weights of multiplicities for separation points $s+1$ and $s$ is
\[2,2,2,\ldots,2,-1,-2,-2,\ldots,-2,-2,\]
where $-1$ is the multiplicity for the coin of type $s+1$. Thus,
\[V(s+1,n) - V(s,n) = 2(1+2+3 +\cdots+s) -(s+1) - 2((s+2)+(s+3)+\cdots +n).\]
Simplifying, we get
\begin{multline*}
V(s+1,n) - V(s,n) = s(s+1) -(s+1) - (n+s+2)(n-s-1) = \\
 s^2 -1 -n^2 -n(s+2) + n(s+1) +(s+2)(s+1) = -n^2 - n + (s+1)(2s+1).
\end{multline*}

We see that, for small $s$, this expression is decreasing until $s$ reaches the value that satisfies the equation
\begin{equation}\label{equation:criticalPoint}
(s+1)(2s+1) = n^2+n.
\end{equation}
Then, it starts increasing. Let us denote the root of Equation~\ref{equation:criticalPoint} by $s'$. We see that $s' \approx \frac{\sqrt2 n}{2}$. The weight of the left pan in a non-tight weighing with separation point $s$ is at least $2W_L(s,n)$. The weight on the right is at least $2W_R(s,n) - s - (s+1) - \cdots - n$. Thus, the total weight is at least 
\begin{equation}\label{equation:lowerboundV}
V(s,n) = 2W_L(s,n) + 2W_R(s,n) - s - (s+1) - \cdots - n.
\end{equation}

We chose $s'$ to be the minimum of $V(s,n)$ over $s$. This means that the total weight of a non-tight downhill weighing is at least $V(s',n)$. Now we notice that $s' > \frac{2n}{3}$. To prove this, assume for the sake of contradiction that $s' \le \frac{2n}{3}$. Then,
\[(s' + 1)(2s' + 1) \le \left(\frac{2n}{3} + 1\right)\left(\frac{4n}{3} + 1\right) = \frac{8}{9}n^2 + 2n + 1.\]
Moreover, 
\[\frac{8}{9}n^2 + 2n + 1 < n^2 + n\] for all $n > \frac{9 + 3\sqrt{13}}{2} \approx 9.90$. Thus, for all $n \ge 10$, we must have $s' > \frac{2n}{3}$.

This means that 
\[V(s',n) > 4W_L(\frac{2n}{3},n) = 4\binom{\frac{2n}{3}}{3}.\]

We showed in Section~\ref{sec:data} that the minimum weight for balanced and tight imbalanced weighings does not exceed 
\[\frac{8n^3+56n^2+9n-8}{81}.\]

We now show that our lower bound for $V(s', n)$ is never better than this upper bound. First, we have 
\[V(s',n) > 4\binom{\frac{2n}{3}}{3} = 4 \frac{\frac{2n}{3}(\frac{2n}{3} - 1)(\frac{2n}{3} - 2)}{6} > \frac{8n^3+56n^2+9n-8}{81}\]
whenever 
\[D(n) = 4 \frac{\frac{2n}{3}(\frac{2n}{3} - 1)(\frac{2n}{3} - 2)}{6} - \frac{8n^3+56n^2+9n-8}{81} = \frac{8 n^3 - 128 n^2 + 63 n + 8}{81} > 0.\]
Since the leading coefficient of $D$ is positive, and the largest root of $D$ is approximately $15.48$, we have $D(n) > 0$ for all integers $n$, where $n \ge	16$.

This means that, for all $n \ge 16$, non-tight imbalances are strictly worse in terms of weight. 

We now resolve the cases when $3 \le n \le 15$ ($n = 1$ does not require using the scale and non-tight weighings are clearly worse for $n = 2$). Recall that $s'$ is the root of Equation~\ref{equation:criticalPoint}, which, when solved explicitly, yields one positive root, being $$s' = \frac{-3 + \sqrt{1 + 8 n + 8 n^2}}{4}.$$
		
For $n = 3, 4, 5, 6, 7, 8, 9, 10, 11, 12, 13, 14, 15$, the corresponding values of $s'$ are approximately \newline $1.71, 2.42, 3.13, 3.84, 4.55, 5.26, 5.96, 6.67, 7.38, 8.09, 8.79, 9.50, 10.21$. 
		
Note that we can approximate Equation~\ref{equation:lowerboundV} for real $s$ by using $$V(s, n) = 2\frac{s(s - 1)(s - 2)}{6} + 2\frac{(s - n - 1)(s - n - 2)(s + 2n)}{6} - \frac{(n + s)(n - s + 1)}{2}$$.
		
Plugging our values of $s'$ into this approximation yields Table~\ref{table:notTightVersusPrevBound}. 

\begin{table}[htb]
\centering
\begin{tabular}{| c |c| c|c|c|}
	\hline
	$n$  & $s'$ & $\ceil{V(s', n)}$ & $W_M(n)$\\
	\hline
	$3$ & $1.71$ & $14$ & $5$\\
	$4$ & $2.42$ & $25$ & $8$\\
	$5$ & $3.13$ & $40$ & $20$\\
	$6$ & $3.84$ & $61$ & $33$\\
	$7$ & $4.55$ & $89$ & $40$\\
	$8$ & $5.26$ & $126$ & $75$\\
	$9$ & $5.96$ & $172$ & $99$\\
	$10$ & $6.67$ & $228$ & $112$\\
	$11$ & $7.38$ & $297$ & $168$\\
	$12$ & $8.09$ & $378$ & $219$\\
	$13$ & $8.79$ & $474$ & $240$\\
	$14$ & $9.50$ & $585$ & $337$\\
	$15$ & $10.21$ & $712$ & $409$\\
	\hline
\end{tabular}
\caption{Non-tight weighings compared to previous upper bound}\label{table:notTightVersusPrevBound}
\end{table}
		
Note that $W_M(n) < \ceil{V(s', n)}$ for all $3 \le n \le 15$. Thus, we have now shown that non-tight weighings are never better than our previous upper bound $W_M(n)$) over all $n$.
\end{proof}

\section{Solo weighings}\label{sec:solo}

In this section, we consider the case when the right-hand side has only one type of coin. We call such weighings \textit{solo weighings}.

We want to find downhill verifying solo weighings such that the multiplicities on the left side are consecutive numbers. As we are mostly interested in weighings with smaller weights and fewer number of coins, we consider two possibilities: the range on the left is $[0\ldots n-2]$, or it is $[1\ldots n-1]$.

\begin{lemma}
A balanced downhill verifying solo weighing such that the multiplicities on the left pan are consecutive numbers in the range $[1\ldots n-1]$ exists if and only if $n \equiv \pm 1 \pmod{6}$. An imbalanced downhill verifying solo weighing such that the multiplicities on the left pan are consecutive numbers in the range $[1\ldots n-1]$ exists if and only if $n = 2$, or $n=6$.
\end{lemma}

\begin{proof}
Consider the multiplicities on the left pan: $n - 1, n - 2, n - 3, \ldots 2, 1$ for a total weight $\binom{n + 1}{3}$.

If the weighing is a balance, then $n | \binom{n + 1}{3}$. For ratio $\binom{n + 1}{3} : n = \frac{(n + 1)(n - 1)}{6}$ to be an integer, number $n$ must be congruent to $1$ or $5$ mod 6. 

If the weighing is an imbalance, then, for $n > 2$, it must be a tight imbalance. Consequently, the RHS must have a weight of $\binom{n + 1}{3} + 1$, and thus, for this weighing to be possible, we must have $n | \binom{n + 1}{3} + 1$. Note that this is true if and only if $$\frac{\binom{n + 1}{3} + 1}{n} = \frac{\frac{(n + 1)(n)(n - 1)}{6} + 1}{n} = \frac{(n + 1)(n - 1)}{6} + \frac{1}{n}$$ is an integer. For $n \leq 6$, we see that $n = 1, 2, 6$ gives integers for this expression. However, for $n  >6$, this expression is never an integer because the fractional part of $\frac{(n - 1)(n + 1)}{6}$ is one of: $0, \frac{1}{6}, \frac{1}{3}, \frac{1}{2}, \frac{2}{3}, \frac{5}{6}$. Adding $\frac{1}{n}$ for $n > 6$ cannot make this to an integer. 
\end{proof}

In the case of a balance in the above lemma, the total weight is $2\binom{n + 1}{3}$ and the total number of coins is 
$$\frac{(n - 1)n}{2} (\text{on the LHS})  + \frac{(n + 1)(n - 1)}{6} (\text{on the RHS}) = \frac{4n^2 - 3n - 1}{6}.$$

For example, for 5 bags, we have multiplicities $\{4, 3, 2, 1, -4\}$ with 14 coins, and the total weight is 40. 

Notice that this solo weighing is more efficient than our first example, when we had only one type of coin on the left pan.
 
\begin{lemma}
A downhill verifying solo weighing such that the multiplicities on the left pan are consecutive numbers in the range $[0\ldots n-2]$ exists if and only if $n \equiv \pm 1 \pmod{3}$.
\end{lemma}

\begin{proof}
Consider the multiplicities on the LHS: $n - 2, n - 3, n - 4, \ldots, 2, 1$ for the total weight of $\binom{n}{3}$. 

If the weighing is balanced, then $n| \binom{n }{3}$. For the ratio $\binom{n}{3} : n = \frac{(n - 1)(n - 2)}{6}$ to be an integer, number $n$ must be congruent to $1$ or $2$ mod 3.

If the weighing is imbalanced, then, for $n > 2$ it must be a tight imbalance. Consequently, the RHS must have a weight of $\binom{n}{3} + 1$. Thus, for this weighing to be possible, we must have $n | \binom{n}{3} + 1$. Note that this is true if and only if $$\frac{\binom{n}{3} + 1}{n} = \frac{\frac{(n)(n - 1)(n - 2)}{6} + 1}{n} = \frac{(n - 1)(n - 2)}{6} + \frac{1}{n}$$ is an integer. For $n \leq 6$, this expression is only an integer for $n = 1$. For $n > 6$, this expression is never an integer for the same reasons as in the previous lemma.
\end{proof}

The total weight on both sides is $2\binom{n}{3}$, and the total number of coins is $$\frac{(n - 2)(n - 1)}{2} (\text{on the LHS})  + \frac{(n - 1)(n - 2)}{6} (\text{on the RHS}) = \frac{2}{3}(n - 1)(n - 2).$$

For example, for 5 bags, we have multiplicities $\{3, 2, 1, 0, -2\}$ with 8 coins, and the total weight is 20.

\section{Multiplicities are in an arithmetic progression}\label{sec:ap}

In this section, we discuss verifying balanced weighings that form an arithmetic progression. We are interested only in primitive weighings. We denote the difference in this arithmetic progression as $d$.

\begin{lemma}
A balanced verifying primitive weighing forms an arithmetic progression if and only if either $d=3$ and $n$ could be any integer, or $d=1$ and $n=3k+1$.
\end{lemma}

\begin{proof}
Assume our multiplicities are
\[a,\ a-d,\ a-2d,\ \ldots,\ a-(n-1)d,\]
where $a$ is the first multiplicity. As our weighing is primitive, we have $\gcd(a,d)=1$.

As the weighing balances, we know that 
\[a + 2(a-d) + 3(a-2d) + \cdots + n(a-(n-1)d)=0.\]
We can group all the $a$'s and $d$'s together to get:
\begin{align*}
   0 &=a(1+2+3+\cdots +n) - d(2 \cdot 1 + 3 \cdot 2 + \cdots + n(n-1)) \\
    &= a\left(\cfrac{n(n+1)}{2}\right) - d\left(\cfrac{(n-1)n(n+1)}{3}\right). \\
\end{align*}
Dividing by $\frac{n(n+1)}{2}$ yields:
\[a = \cfrac{2}{3}(n-1)d.\]

Since $a$, $d$, and $n$ are all integers, $n \equiv 1$ (mod 3) or $d \equiv 0$ (mod 3). In addition, as $\gcd(a,d)=1$, if $d \equiv 0$ (mod 3), then $d=3$, and if $n \equiv 1$ (mod 3), then $d=1$. Therefore, we can only get primitive verifying balances  in an arithmetic progression when either $d= 3$, or $d=1$ and there are $3k+1$ bags.
\end{proof}

For a small number of bags, the weighings with a difference of 3 are as follows: 
\begin{center}
\begin{tabular}{ |c|c| } 
 \hline
 Number of Bags: & Weighing: \\
 \hline
 3 & (4, 1, -2) \\
 \hline
  4 & (6, 3, 0, -3) \\
 \hline
  5 & (8, 5, 2, -1, -4) \\
 \hline
  6 & (10, 7, 4, 1, -2, -5) \\
 \hline
  7 & (12, 9, 6, 3, 0, -3, -6) \\
 \hline
  8 & (14, 11, 8, 5, 2, -1, -4, -7) \\
 \hline
\end{tabular}
\end{center}

We can notice that the multiplicities are of the general form
\[2n-2, 2n-5, 2n-8, \ldots, 1-n,\]
where the next multiplicity is the previous one minus 3 and $n\geq 3$.

\section{Acknowledgements}

This project was done as part of MIT PRIMES STEP, a program that allows students in grades 6 through 9 to try research in mathematics. Tanya Khovanova is the mentor of this project. We are grateful to PRIMES STEP and to its director, Slava Gerovitch, for this opportunity.

\end{document}